\documentclass[letterpaper, 10 pt, conference]{ieeeconf}  
\usepackage{cite}
\usepackage{amsmath,amssymb,amsfonts}
\usepackage{algorithm2e}
\usepackage{mathtools}
\usepackage{algpseudocode}
\usepackage{graphicx}
\graphicspath{{photos/}}
\usepackage{color}
\usepackage{algpseudocode}
\usepackage{hyperref}
\usepackage{textcomp}
\usepackage{subfigure}
\usepackage{epsfig} 

\IEEEoverridecommandlockouts                              
\title{\LARGE \bf
Koopman Data-Driven Predictive Control with Robust Stability and Recursive Feasibility Guarantees
}

\author{Thomas de Jong$^{1}$, Valentina Breschi$^{1}$, Maarten Schoukens$^{1}$ and Mircea Lazar$^{1}$
\thanks{$^{1}$~Department of Electrical Engineering,
        Eindhoven University of Technology, 5612 AZ Eindhoven, The Netherlands. E-mails of the authors: \texttt{t.o.d.jong@tue.nl, v.breschi@tue.nl, m.schoukens@tue.nl, m.lazar@tue.nl}.}%
}

\newcommand{\Nset}{\mathbb{N}}
\newcommand{\Rset}{\mathbb{R}}
\newcommand{\Eset}{\mathbb{E}}
\newcommand{\Zset}{\mathbb{Z}}
\newcommand{\Uset}{\mathbb{U}}
\newcommand{\Xset}{\mathbb{X}}

\newcommand{\Yset}{\mathbb{Y}}

\newcommand{\cK}{\mathcal{K}}

\newcommand{\cL}{\mathcal{L}}

\newcommand{\tini}{\text{ini}}
\newcommand{\bu}{\mathbf{u}}
\newcommand{\be}{\mathbf{e}}
\newcommand{\bx}{\mathbf{x}}
\newcommand{\bz}{\mathbf{z}}

\newtheorem{theorem}{Theorem}[section]
\newtheorem{lemma}[theorem]{Lemma}
\newtheorem{problem}[theorem]{Problem}

\newtheorem{definition}[theorem]{Definition}
\newtheorem{remark}[theorem]{Remark}
\newtheorem{assumption}[theorem]{Assumption}

\newcommand{\col}{\operatorname{col}}

\newcommand{\figurename}[1]{Fig.#1}

\begin{document}

\maketitle
\thispagestyle{empty}
\pagestyle{empty}

\begin{abstract}
In this paper, we consider the design of data-driven predictive controllers for nonlinear systems from input-output data via linear-in-control input Koopman lifted models. Instead of identifying and simulating a Koopman model to predict future outputs, we design a subspace predictive controller in the Koopman space. This allows us to learn the observables minimizing the multi-step output prediction error of the Koopman subspace predictor, preventing the propagation of prediction errors. To avoid losing feasibility of our predictive control scheme due to prediction errors, we compute a terminal cost and terminal set in the Koopman space and we obtain recursive feasibility guarantees through an interpolated initial state. As a third contribution, we introduce a novel regularization cost yielding input-to-state stability guarantees with respect to the prediction error for the resulting closed-loop system. The performance of the developed Koopman data-driven predictive control methodology is illustrated on a nonlinear benchmark example from the literature.   
\end{abstract}

\section{INTRODUCTION}
By leveraging the linearity of the predictor in the lifted Koopman space, linear-in-control input Koopman Model Predictive Control (MPC) represents a promising and appealing approach to nonlinear MPC. Indeed, the shift to the Koopman space allows one to convert the nonlinear predictive control problem into a quadratic program (QP), which can then be solved effectively and with high precision. This approach was originally proposed in \cite{korda2018linear}, where the Koopman operator is extended to controlled dynamical systems using Extended Dynamic Mode Decomposition (EDMD) and then used for predictive control purposes. It has then been extended to consider full-state measurements and input-output scenarios, as discussed in \cite{korda2018linear,korda2020koopman}.

Despite its advantages, approximating a nonlinear system with finite-dimensional linear-in-control input Koopman modeling has its limitations. Indeed, it leads to models that can be incomplete, as shown in \cite{iacob2024koopman}. Moreover, while the Koopman operator provides coordinate transformations that extend the local neighborhoods where a linear model is valid to the full basin around them \cite{lan2013linearization}, these models cannot capture the global dynamics of nonlinear systems with multiple equilibria, as any linear operator has either one fixed point at the origin or a subspace of infinitely many fixed points \cite{brunton2021modern,brunton2016koopman}.

This inherent incompleteness of linear-in-control input Koopman models results in discrepancies between real-life nonlinear system outputs and Koopman model predictions, which can be detrimental to predictive control design. In particular, it is no longer possible to guarantee stability and recursive feasibility using just terminal ingredients in the Koopman space. A solution to this problem based on robust tube MPC design in the Koopman space was recently presented in \cite{zhang2022robust}. While offering robust stability and recursive feasibility guarantees to Koopman model-based predictive control, this approach relies on tightened constraints in the Koopman space, which might lead to arbitrary conservatism as constraints on the observables do not relate linearly to the original constraints on outputs or states. Moreover, it requires the measurability of the controlled system state and it postulates the availability of an accurate Koopman model. 

These models can be identified from data using system identification methods, that however typically look at fitting one-step linear state-space dynamics, i.e., the system matrices \cite{korda2018linear,Colin2019,borghi2021koopman}. As a result, prediction errors propagate as the prediction horizon increases. This limitation is overcome in data-driven predictive control (DPC) methods, whose aim is to minimize predicted tracking errors and control efforts over the entire prediction horizon \cite{verheijen2023handbook, kohler2022state}. 

A recent observation in \cite{lazar2023basis} intertwines linear-in-control input Koopman Data-driven Predictive Control and DPC, highlighting that the former can be reformulated as subspace predictive control (SPC) in the Koopman space. The SPC approach to data-driven predictive control comes with several benefits. It removes the propagation of prediction errors, as separate multi-step predictors are identified jointly. The resulting prediction matrices are least squares optimal, yielding unbiased predictors. Last but not least, the resulting predictors are linear in the predicted control input sequence.

Motivated by these insights, in this work we propose a linear-in-control input SPC formulation in the Koopman space, which consists of three main design steps. Firstly, we parameterize the observables as basis or kernel functions that depend on past input and past/current output data of the nonlinear system. We solve a nonlinear least squares problem to learn the observables, aiming to minimize the multi-step output prediction error of the Koopman SPC predictors. It is important to note that our approach differs from that in reference \cite{iacob2021deep}. While \cite{iacob2021deep} identifies a model that requires iterating the Koopman state-space equations $N$ times to obtain the $N$-step ahead prediction, our method eliminates the need for such iteration for multi-step ahead predictions. Thirdly, we use least squares identification to compute the optimal multi-step prediction matrices allowing us to predict the evolution of the full Koopman lifted state, rather than the output, over the prediction horizon. This last choice is advantageous because it enables the computation of a terminal cost and terminal set in the Koopman space. At the same time, since the linear-in-control input Koopman predictors are still incomplete, the corresponding prediction errors might still result in a loss of feasibility despite using a terminal set in the Koopman space. To circumvent this issue, inspired by the solution proposed in \cite{kohler2022recursively} for robust stochastic MPC, we define the initial state of the Koopman model as the interpolation of the measurements based lifted state and the previously one-step ahead predicted Koopman state. We prove that this choice leads to a recursively feasible Koopman data-driven predictive control scheme, despite the prediction error. Based on this definition of the initial state, we also develop a novel regularization cost and establish that the original nonlinear system in closed-loop with the Koopman data-driven predictive controller is input-to-state stable (ISS) \cite{JiangWang01} with respect to the prediction error.

The remainder of the paper is organized as follows. After introducing the preliminaries and our problem statement in Section~\ref{sec:preliminaries}, we present methods for learning 
observables within the Koopman formulation in Section~\ref{sec:koopmanModelIdentification}. To address potential plant-model mismatches stemming from the linear-in-control input Koopman formulation, we develop robust stability and recursive feasibility conditions in Section~\ref{sec:RobustStability}. Finally, in Section~\ref{sec:IllustraticeExample}, we substantiate the effectiveness of our proposed control scheme through a nonlinear example.
\paragraph*{Notation and basic definitions} Let $\mathbb{R}$, $\mathbb{R}_{+}$ and $\mathbb{N}$ denote the set of real numbers, the set of non-negative reals, the set of non-negative integers. For every $c \in \mathbb{R}$ and $\Pi \subseteq \mathbb{R}$ define $\Pi_{\geq c} := \{k \in \Pi | k \geq c \}$ and let $\textbf{w}_{[0,k]}:= \{w(l)\}_{l\in\Zset_{[0,k]}}$. Throughout this paper, for any finite number $q\in\mathbb{N}_{\geq 1}$ of column vectors or functions $\{\xi_1,\dots,\xi_q\}$ we will make use of the operator $\col(\xi_1,\dots,\xi_q) := [\xi_1^\top,\dots,\xi_q^\top]^\top$ and for a square matrix $A$ we use the operator blkdiag($n,A$) defined as the the block diagonal matrix with $n$ matrices $A$ on the diagonal. We denote by $\rho(A)$ for some square matrix $A$ the maximum absolute value of the eigenvalues of $A$. For $c_1\in\Rset_{>0}$, a function $\varphi : \Rset_{[0,c_1]}\rightarrow\Rset_+$ \emph{belongs to class} $\cK$ if it is continuous, strictly increasing and $\varphi(0)=0$. A function $\varphi:\Rset_+\rightarrow\Rset_+$ \emph{belongs to class} $\cK_{\infty}$ if $\varphi\in\cK$ and $\lim_{s\rightarrow\infty}\varphi(s)=\infty$. A function $\beta:\Rset_+\times\Rset_+\rightarrow\Rset_+$ \emph{belongs to class} $\cK\cL$ if for each fixed
$k\in\Rset_+$, $\beta(\cdot,k)\in\cK$ and for each fixed $s\in\Rset_+$, $\beta(s,\cdot)$ is
decreasing and $\lim_{k\rightarrow\infty}\beta(s,k)=0$. Given a signal $v \in \mathbb{R}^{n_v}$, a starting time $k \geq 0$ and and ending instant $j\geq k+1$, we define
\begin{equation}\label{eq:single_Hankel}
        \Bar{\textbf{v}}(k,j) = \col{(v(k),\dots,v(k+j-1))}.
\end{equation}
\section{Preliminaries \& problem statement}\label{sec:preliminaries}

We consider nonlinear MIMO systems with inputs $u\in\mathbb{R}^m$ and measured outputs $y\in\mathbb{R}^p$, whose dynamics can be represented with an \emph{unknown} discrete--time state--space model:
\begin{equation}\label{eqn:state_space}
    \begin{aligned}
        & x(k+1)= f(x(k),u(k)), \quad k\in\mathbb{N}, \\
        &y(k)= h(x(k)),
    \end{aligned}
\end{equation}
$x\in\mathbb{X}$ is an unmeasured state, and the functions $f:\mathbb{R}^{n} \times \mathbb{R}^{m} \rightarrow \mathbb{R}^{n}$ and $h:\mathbb{R}^{n} \rightarrow \mathbb{R}^{p}$ are assumed to be \emph{unknown}. We further assume that the following properties hold.
\begin{assumption}\label{assumption:contobs}
    The system~\eqref{eqn:state_space} is controllable and observable.
\end{assumption}
\begin{assumption}\label{assumption:origin}
    The origin is a stabilizable equilibrium point for system~\eqref{eqn:state_space}.
\end{assumption}
Consider next a closed-loop counterpart of system~\eqref{eqn:state_space}, assuming that the input is a perturbed state feedback law $u(k) = \phi(x(k),e(k))$, i.e.,
\begin{equation}\label{eqn:state_space_preturbed}
    \begin{aligned}
        &x(k+1) = f(x(k),\phi(x(k),e(k))), \quad k\in\mathbb{N}, \\
        &y(k) = h(x(k)),
    \end{aligned}
\end{equation}
where $e:\Nset\rightarrow\Rset^{n_e}$ is an unknown \emph{inner} perturbation trajectory. We can define input-to-state stability for the closed-loop system in \eqref{eqn:state_space_preturbed} as follows.
\begin{definition}\label{def:ISS}
    We call the system~\eqref{eqn:state_space_preturbed} input-to-state stable in $\Xset$ for perturbations in $\Eset$, or shortly $ISS(\Xset,\Eset)$, if there exists a $\cK\cL$-function $\beta$ and a $\cK$-function $\gamma$ such that, for each $x(0)\in\mathbb{X}$ and all $\textbf{e} = \{e(l)\}_{l\in\Zset_+}$, it holds that the corresponding state trajectory of \eqref{eqn:state_space_preturbed} satisfies 
    \[\|x(k)\|\leq \beta(\|x(0)\|,k) + \gamma(\|\textbf{e}_{[0,k]}\|), \quad \forall k \in \Zset_{\geq1}. \]
\end{definition}\medskip
In this work, we exploit the fact that we can approximate a nonlinear system of the form \eqref{eqn:state_space} by a higher dimensional linear model utilizing the Koopman framework \cite{koopman1931hamiltonian}, where the state $x(k)$ of the original system \eqref{eqn:state_space} is lifted to a higher dimensional space via a set of observables, i.e.,
\begin{equation*}
    z(k) := \varphi(x(k)) := \col(\varphi_1(x(k),\dots,\varphi_L(x(k)),
\end{equation*}
and $z\in\mathbb{R}^L$ with $L>n$. A a commonly followed approach is to use an approximate Koopman model for MPC \cite{korda2018linear} which is the linear-in-control input embedding, i.e.,
\begin{equation}\label{eqn:koopman}
    \begin{aligned}
        & z(k+1) = A z(k) + B u(k), \quad k\in\mathbb{N}, \\
        &z(0) = \varphi(\textbf{x}_\tini(k)),
    \end{aligned}
\end{equation}
where we use an input-output data embedding of the state as in \cite{korda2020koopman}. This model is preferred even though it is incomplete, because it leads to a QP problem to be solved online. 
\begin{remark}In \cite{iacob2024koopman} it was shown that for general discrete-time nonlinear systems of the form \eqref{eqn:state_space} in combination with a finite set of continuously differentiable observable functions $\varphi\in\Xset\rightarrow\Rset^L$ where $\Xset$ is convex such that $\varphi(f(\cdot,0))\in\text{span}\{\varphi\}$, then there exists an exact finite dimensional lifting
\begin{equation}
\label{eq:Cristi}
\begin{split}
    &\varphi(x(k+1)) = \mathcal{A} \varphi(x(k)) + \mathcal{B}(x(k),u(k))u(k).
\end{split}
\end{equation}
For alternative complete nonlinear Koopman embedding methods we refer to \cite{goswami2017global,proctor2018generalizing}. It would be of interest to research further usage of such embedding methods in Koopman MPC. 
\end{remark}

\begin{remark}\label{remark:linearity}
Even if \cite{iacob2024koopman} proves that linear-in-control input Koopman models as \eqref{eqn:koopman} are incomplete, hence leading to considerable prediction errors, for relatively short prediction horizons $N$ the system remains relatively close to the initial conditions $z(0)$.In such a case it can be assumed that the transient response dominates the forced response. This further supports the validity of using \eqref{eqn:koopman} for predictive control strategies.
\end{remark}

Above in \eqref{eqn:koopman}, $\textbf{x}_\tini(k):=\col(\textbf{u}_\tini(k),\textbf{y}_\tini(k))$ is a collections of $T_\tini$ past inputs and outputs defined as 
\begin{equation*}
    \begin{split}
        \textbf{u}_\tini(k) &:= \col(u(k-T_\tini+1),\dots,u(k-1)) \in\mathbb{R}^{T_\tini m}, \\
        \textbf{y}_\tini(k) &:= \col(y(k-T_\tini+1),\dots,y(k))\in\mathbb{R}^{T_\tini p},        
    \end{split}
\end{equation*}
while the observables satisfy the following assumption.
\begin{assumption}\label{assumption:observables}
   The functions $\varphi(\textbf{x}_{ini}(k))$ parameterizing the vectors of observables are continuous and are such that $\varphi(0) = 0$.
\end{assumption}
The Koopman model \eqref{eqn:koopman} can can be used to define $N$-step ahead prediction matrices: 
\begin{equation}\label{eqn:multi-step}\Psi(A) :=
    \begin{bsmallmatrix}
        A \\
        A^2 \\
        \vdots \\
        A^N
    \end{bsmallmatrix}, \Gamma(A,B) :=
     \begin{bsmallmatrix}
        B & 0 & \dots & 0 \\
        AB & B & \dots & 0 \\
        \vdots & \vdots & \ddots & \vdots \\
        A^{N-1}B & A^{N-2}B & \dots & B \\
    \end{bsmallmatrix},
\end{equation}
ultimately allowing one to follow the \emph{indirect} data-driven predictive control, see e.g., \cite{masti2020learning}, where typically multi-step predictors of the NARX type \cite{billings2013nonlinear} are identified from the input-output data generated by system \eqref{eqn:state_space} and used to predict $\textbf{y}_{[1,N]}(k)$ and $\textbf{z}_{[1,N]}(k)$  form past inputs and outputs $\textbf{u}_\tini(k)$, $\textbf{y}_\tini(k)$ and $\textbf{u}_{[0,N-1]}(k)$, with, 
\begin{equation*}
    \begin{aligned}
        &\textbf{u}_{[0,N-1]}(k) := \col(u_{0|k},\dots,u_{N-1|k}), \\
        &\textbf{y}_{[1,N]}(k) := \col(y_{1|k},\dots,y_{N|k}), \\
        &\textbf{z}_{[1,N]}(k) := \col(z_{1|k},\dots,z_{N|k}). 
    \end{aligned}
\end{equation*}

\subsection{Problem Statement}
Assume that we are given a set of noiseless input-output data $\{y(i),u(i)\}_{i\in[0,s]}$ gathered from the unknown system \eqref{eqn:state_space}, where $s\in\Nset$ denotes the number of samples. In this context, our goal is the following. 
\begin{problem}\label{Problem:1}
Given a finite set of noiseless input-output data $\{y(i),u(i)\}_{i\in[0,s]}$ 
the objective is to utilize them \emph{exclusively} to construct an optimal linear-in-control input multi-step prediction model of the form:
\begin{equation}\label{eqn:KoopmanMultistep}
\textbf{z}_{[1,N]}(k) = \Psi z_{0|k} + \Gamma \textbf{u}_{[0,N-1]}(k),
\end{equation}
minimizing the prediction error 
\begin{equation}\label{eqn:prediction_error}
e(k) := z_{1|k-1} - \varphi(\bx_\tini(k)),\quad \forall k\in\Nset.
\end{equation}\hfill $\square$
\end{problem}
We approach this problem by developing a multi-step predictor that minimizes the prediction error over the entire prediction horizon, to ensure prediction errors do not propagate. Due to the availability of only input and output data, we adopt a two-step approach to learn a state space model in the Koopman lifted space, elaborated upon in Section~\ref{sec:koopmanModelIdentification}.

Since the linear-in-control input Koopman form is incomplete, using it for control purposes further leads to the problem formalized in the following statement. 
\begin{problem}\label{problem:2}
    Given a nonlinear system of the form \eqref{eqn:state_space} and a non-exact multi-step Koopman predictor \eqref{eqn:KoopmanMultistep} obtained by solving Problem~\ref{Problem:1}, develop a data-driven predictive control scheme with recursive feasibility and robust stability guarantees.  \hfill $\square$
\end{problem}

We approach this problem using an interpolated initial condition \cite{kohler2022recursively} in the data-driven predictive control problem, as elaborated in Section~\ref{sec:koopmanModelIdentification}. This approach allows us to guarantee recursive feasibility, but it makes guaranteeing stability more challenging. To circumvent this issue,  we develop robust stability (ISS) conditions (see Section~\ref{sec:RobustStability}).

\section{Multi-step learning of the\\ Koopman observables and prediction matrices}\label{sec:koopmanModelIdentification}
Toward framing our approach to tackle Problem~\ref{Problem:1}, let us initially define the following data Hankel matrices:
\begin{equation}\label{eqn:hankel}
    \begin{split}
        \textbf{U}_p &:= \begin{bmatrix} \Bar{\textbf{u}}(0,T_\tini-1), \dots, \Bar{\textbf{u}}(T,T_\tini-1)\end{bmatrix}, \\
        \textbf{Y}_p &:= \begin{bmatrix} \Bar{\textbf{y}}(0,T_\tini), \dots, \Bar{\textbf{y}}(T,T_\tini)\end{bmatrix}, \\
         \textbf{Z}_p &:= \begin{bmatrix} \Bar{\textbf{z}}(0,1), \dots, \Bar{\textbf{z}}(T,1)\end{bmatrix}, \\
        \textbf{U}_f &:= \begin{bmatrix} \Bar{\textbf{u}}(T_\tini-1,N), \dots, \Bar{\textbf{u}}(T_\tini-1+T ,N)\end{bmatrix}, \\
        \textbf{Y}_f &:= \begin{bmatrix} \Bar{\textbf{y}}(T_\tini,N), \dots, \Bar{\textbf{y}}(T_\tini+T,N)\end{bmatrix}, \\
        \textbf{Z}_f &:= \begin{bmatrix} \Bar{\textbf{z}}(1,N), \dots, \Bar{\textbf{z}}(T+1,N)\end{bmatrix}, \\
    \end{split}
\end{equation}
where $T\geq(m+p)T_\tini+mN$ is the number of their columns. 

Since it is not possible to capture the global dynamics of nonlinear systems with multiple equilibria using a single linear operator \cite{brunton2021modern,brunton2016koopman}, we tackle Problem~\ref{Problem:1} under Assumption~\ref{assumption:contobs} and by leveraging Assumption~\ref{assumption:origin} to fix the origin as the equilibrium point of system~\eqref{eqn:state_space}. 

To identify directly a multi-step output predictor from the available input-output data,  we define 
\begin{equation}
        \Tilde{\Psi} := \Tilde{C} \Psi, \quad \quad \Tilde{\Gamma} := \Tilde{C} \Gamma, \\
\end{equation}
with the  block diagonal matrix $\Tilde{C} := \text{blkdiag}(N,C)$. Accordingly, our goal of minimizing the prediction error translates into the following loss function: 
\[
J(\textbf{Y}_f, \hat{\textbf{Y}}_{f}) := \|\textbf{Y}_f- \hat{\textbf{Y}}_{f}\|_2^2,
\]
where $\hat{\textbf{Y}}_{f}$ is the Hankel matrix of future outputs generated by the linear-in-control multi-step Koopman model predictor. Hence, the multi-step learning problem we aim to solve can be formalized as:
\begin{equation}
\label{eq:fit_basis}
\begin{aligned}
&\min_{\{\varphi,\Tilde{\Psi} ,\Tilde{\Gamma}\}} J(\textbf{Y}_f, \hat{\textbf{Y}}_{f}) \\ 
&\text{subject to: }  \\
&\qquad \hat{\textbf{Y}}_{f} = \Tilde{\Psi}\varphi\bigg(\begin{bmatrix}
   \textbf{U}_p \\
   \textbf{Y}_p
\end{bmatrix}\bigg) + \Tilde{\Gamma} \textbf{U}_f.
\end{aligned}
\end{equation}
The observables $\varphi_i$ can be chosen as some basis functions (e.g., radial basis functions) or neural networks \cite{lian2021koopman} and, depending on this structural choice, \eqref{eq:fit_basis} can be a nonlinear least squares problem. As an example, this is what happens when parametrizing the observable as radial basis function and learning their centers, or when the observables are parameterized as neural networks with nonlinear activation functions (such as $\tanh$ or sigmoid).  

Solving \eqref{eq:fit_basis} results in a multi-step predictor of the form
 \begin{align*}
     \textbf{y}_{[1,N]} &= \Tilde{\Psi}^*z(k) + \Tilde{\Gamma}^* \textbf{u}_{[0,N-1]}\\
     &=\Tilde{\Psi}^*\varphi^*(\textbf{x}_\tini(k)) + \Tilde{\Gamma}^* \textbf{u}_{[0,N-1]},
 \end{align*} 
 where we have introduced the definition of the state at time $k$ according to \eqref{eqn:koopman}. 
 
 After this first identification step, we propose to refine the estimates of $\Tilde{\Psi}$ and $ \Tilde{\Gamma}$ by exploiting the same arguments of linear subspace identification. Specifically, we feed the input-output data through the learned observables $\varphi^*$, construct the Koopman lifted state Hankel matrices $\textbf{Z}_p$ an $\textbf{Z}_f$ from \eqref{eqn:hankel} and, then, compute the least-squares estimates
 \begin{equation}\label{eqn:LS_solution}
     \textbf{Z}_{f} \begin{bmatrix}
     \textbf{Z}_p  \\
\textbf{U}_{f}
\end{bmatrix}^{\dag} = \begin{bmatrix}
    \Psi^{LS} & \Gamma^{LS}
\end{bmatrix}.
 \end{equation}
 Note that, this problem has a unique solution if the matrix $\begin{bmatrix}
\textbf{Z}_p^\top  & \textbf{U}_{f}^\top \end{bmatrix}^\top$ has full row rank. This can be achieved by choosing a persistently exciting input of sufficient order and choosing the experiment length $T$ sufficiently long. The resulting multi-step state predictor is given by
\begin{equation*}\label{eqn:koopman_pm}
    \textbf{z}_{[1,N]}(k) = \Psi^{LS} z_{0|k} + \Gamma^{LS}  \textbf{u}_{[0,N-1]}(k),
\end{equation*}
being a multi-step linear-in-control input Koopman prediction model for the system~\eqref{eqn:state_space}. From $\Psi^{LS}$ and $\Gamma^{LS}$, we can further extract  their first row blocks to construct the matrices $\Tilde{A}$ and $\Tilde{B}$ such that  
\begin{equation*}\label{eqn:koopman_pm}
    z_{1|k} = \Tilde{A} z_{0|k} + \Tilde{B} u_{0}(k),
\end{equation*}
which we use in constructing the terminal ingredients of the proposed predictive control scheme.


\section{Koopman data-driven predictive control with interpolated initial state}\label{sec:RobustStability}
As shown in \cite{iacob2024koopman}, linear-in-control input Koopman models frequently lead to plant-model mismatches. When employing the Koopman model in a predictive control scheme, this mismatch could cause a loss of recursive feasibility and stability guarantees even if a terminal set is used in the Koopman space. Inspired by \cite{kohler2022recursively}, we address this challenge by proposing the following Koopman data-driven predictive control formulation with interpolated initial state:
\begin{subequations}
\label{eq:3:1}
\begin{align}
&\min_{\Xi(k)}  \quad   l_N(z_{N|k}) + \sum_{i=0}^{N-1} l(z_{i|k},u_{i|k}) + l_r(z_{0|k})   \label{eq:3:1a} \\ 
&\text{subject to: } \nonumber  \\
&\qquad \bz_{[1,N]}(k) = \Psi^{LS} z_{0|k} + \Gamma^{LS}  \textbf{u}_{[0,N-1]}(k), \label{eq:3:1b}\\
&\qquad z_{0|k} = (1-\xi_k)\varphi(\bx_\tini(k)) + \xi_k z^*_{1|k-1}, \label{eq:3:1c}\\
&\qquad (\textbf{z}_{[1,N]}(k),\textbf{u}_{[0,N-1]}(k))  \in \Xset_z^N \times \Uset^N, \label{eq:3:1d}\\
&\qquad z_{N|k} \in \mathbb{X}_T\subseteq\Xset_z,\label{eq:3:1e}\\
&\qquad \xi_k \in [0,1].\label{eq:3:1f}
\end{align}
\end{subequations}
where $\Xi(k) :=\col(\bz_{[1,N]}(k),\textbf{u}_{[0,N-1]}(k),z_{0|k},\xi_k)$, $z^*_{1|k-1}$ is the one-step ahead optimal predicted state from the previous sampling instance, the terminal cost is $l_N(z_{N|k}):=z_{N|k}^\top Pz_{N|k}$, and 
the stage cost is $l(z_{i|k},u_{i|k}):=z_{i|k}^\top Qz_{i|k}+u_{i|k}^\top R u_{i|k}$. With respect to \cite{kohler2022recursively}, we propose the alternative regularization term 
\begin{align}
\label{eq:3:lxi}
l_r(z_{0|k}):=\lambda\|z_{0|k}-\varphi(\bx_\tini(k))\|_2^2, \quad \lambda>0,
\end{align}
which directly penalizes the deviation of the initial state in the Koopman space from the Koopman state obtained by lifting the most recent measured input-output data $\bx_\tini(k)$. This alternative regularization is instrumental in establishing input-to-state stability guarantees, see, e.g., \cite{JiangWang01, lazar2013}, with respect to the Koopman model error
\begin{equation}
\label{eq:3:eK}
e(k):=z_{1|k-1}^\ast-\varphi(\bx_\tini(k)),
\end{equation}
namely, the difference between the one-step ahead predicted Koopman state and the corresponding measured Koopman state obtained at time $k$. This prediction error is related to the inner perturbation trajectory in \eqref{eqn:state_space_preturbed} which will later be utilized for stability analysis. Note that in \eqref{eq:3:1} the Koopman predicted states are constrained to a compact set $\Xset_z$ within the Koopman space that contains the origin in its interior. Such a set could be constructed by computing the minimum and the maximum of $\varphi(\bx_\tini)$ over $\bx_\tini\in\Uset^{T_\tini}\times\Yset^{T_\tini}$. Moreover, it is worth remarking that our Assumption~\ref{assumption:observables} is consistent with the fact that we assume the origin is an equilibrium of the original nonlinear system. 
\subsection{Properties of proposed scheme}
To establish the properties of \eqref{eq:3:1}
we have to introduce a terminal cost and set in the Koopman space, which are outlined in the following.
\begin{assumption}\label{assm:TerminalCost_and_Set}
There exists a locally stabilizing Koopman state-feedback control law $u(z(k)):=Kz(k)$, a positive definite matrix $P$ and a compact set $\Xset_T$ with the origin in its interior such that:
\begin{subequations}
\label{eq:3:2}
\begin{align}
&(\Tilde{A}+\Tilde{B}K)^\top P(\Tilde{A}+\Tilde{B}K)-P+Q+K^\top R K \succeq 0\label{eq:3:2a}\\
&(\Tilde{A}+\Tilde{B}K)\Xset_T\subseteq\Xset_T,\quad K\Xset_T\subseteq\Uset,\quad \Xset_T\subseteq\Xset_z. \label{eq:3:2b}
\end{align}
\end{subequations}
\end{assumption}
With this assumption in place, we can formalize our first theoretical result, namely the recursive feasibility of \eqref{eq:3:1}.
\begin{theorem}\label{theorem:RecursiveFeasibility}
Suppose that Assumption~\ref{assm:TerminalCost_and_Set} holds. At any time $k\in\Nset$, given $\bx_\tini(k)$ and $z_{1|k-1}^\ast$, suppose that \eqref{eq:3:1} is feasible. Then the problem in \eqref{eq:3:1} is feasible at time $k+1$ for $\bx_\tini(k+1)$ and $z_{1|k}^\ast$. 
\end{theorem}
\begin{proof}
Consider the optimal predicted input sequence $\bu_{[0,N-1]}^*(k)$ at $k$, i.e.,
\[\bu_{[0,N-1]}^*(k) = \{u^*_{0|k},u^*_{1|k},\dots,u^*_{N-2|k}, u^*_{N-1|k}\}.\]
Then, at time $k+1$, choose $\tilde\xi_{k+1}\!=\!1$, which yields $\tilde z_{0|k+1}=z_{1|k}^\ast$ and construct the sub-optimal input sequence:
\[\tilde\bu_{[0,N-1]}(k+1) = \{u^*_{1|k},u^*_{2|k},\dots,u^*_{N-1|k},\tilde u_{N-1|k+1}\},\]
where $\tilde u_{N-1|k+1}:=Kx^*_{N|k}$. This results in a sub-optimal, but feasible initial condition $\tilde z_{0|k+1}$, as the shifted sub-optimal sequence $\tilde\bu_{[0,N-1]}(k+1)$ satisfies all the constraints in \eqref{eq:3:1}, due to Assumption~\ref{assm:TerminalCost_and_Set} and the fact that
\[
\begin{split}
\tilde z_{N|k+1}&=\Tilde{A}^N z_{1|k}^\ast + \begin{bmatrix}\Tilde{A}^{N-1}\Tilde{B} & \ldots & \Tilde{B}\end{bmatrix} \tilde\bu_{[0,N-1]}(\tilde z_{0|k+1})\\
& = \Tilde{A}(\Tilde{A}^{N-1} z_{1|k}^\ast+\Tilde{A}^{N-2}\Tilde{B} u_{1|k}^\ast+\ldots)+ 
\Tilde{B} \tilde u_{N-1|k+1}\\
&=\tilde Az_{N|k}^\ast+ \Tilde{B} \tilde u_{N-1|k+1}= (\Tilde{A}+\Tilde{B}K)z_{N|k}^\ast\in\Xset_T, 
\end{split}\]
which completes the proof.
\end{proof}
\begin{remark}\label{remark:RF_noise}
Theorem~\ref{theorem:RecursiveFeasibility} also holds in the presence of noise under suitable assumptions. If Problem~\ref{eq:3:1} is feasible at time $k$ it is also feasible at time $k+1$ since $\xi=1$ always results in a feasible problem. 
\end{remark}

Next, we prove that the nonlinear system \eqref{eqn:state_space} controlled in closed-loop with the Koopman DPC control law obtained by solving \eqref{eq:3:1} is ISS with respect to the Koopman model one-step ahead prediction error $e(k)$. To this end, let $\Xset_f(N)$ denote the set of feasible initial states in the Koopman space, i.e., $z_{0|k}$, which is positively invariant by Theorem~\ref{theorem:RecursiveFeasibility} for the closed-loop system \eqref{eqn:state_space}-\eqref{eq:3:1}. Moreover, we define the optimal value function at time $k$ corresponding to \eqref{eq:3:1}, i.e., 
\begin{align*}
V(z_{0|k}^\ast)&:=J(\bz_{[1,N]}^\ast(k),\bu_{[0,N-1]}^\ast(k),z_{0|k}^\ast,\xi_k^\ast)\\
&=l_N(z_{N|k}^\ast) + \sum_{i=0}^{N-1} l(z_{i|k}^\ast,u_{i|k}^\ast) + l_r(z_{0|k}^\ast).
\end{align*} 
With that, we can establish the following instrumental result along with the usual reasoning used to establish that the MPC value function is decreasing.
\begin{lemma}
\label{lem:MPC}
Suppose that Assumption~\ref{assm:TerminalCost_and_Set} holds. Then for all $\varphi(\bx_\tini(0))\in\Xset_z$ and $z_{1|-1}^\ast\in\Xset_z$ such that the corresponding $z_{0|0}^\ast\in\Xset_f(N)$, along the trajectories of the closed-loop system \eqref{eqn:state_space}-\eqref{eq:3:1} it holds that:
\begin{equation}
\label{eq:V:z}
V(z_{0|k+1}^\ast)-V(z_{0|k}^\ast)\leq\! -\alpha_Q(\|z_{0|k}^\ast\|)+\sigma_r(\|e(k+1)\|),
\end{equation}
for all $k\in\Nset$ and for $\alpha_Q(s):=\lambda_\text{min}(Q)s^2\in\cK_\infty$ and $\sigma_r(s):=\lambda s^2\in\cK_\infty$.
\end{lemma}
\begin{proof}
Since  $z_{0|0}^\ast\in\Xset_f(N)$, by Theorem~\ref{theorem:RecursiveFeasibility}, \eqref{eq:3:1} remains feasible for all $k\geq 1$. At any time $k+1\in\Nset$ consider the sub-optimal initial state $\tilde z_{0|k+1}=z_{1|k}^\ast$ obtained for $\tilde \xi_{k+1}=1$  and the corresponding sub-optimal state and input sequences $\tilde\bz_{[1,N]}(k+1), \tilde\bu_{[0,N-1]}(k+1)$. Then, by the optimality of $V(z_{0|k+1}^\ast)$ it holds that
\begin{align*}
&V(z_{0|k+1}^\ast)-V(z_{0|k}^\ast)\\
&\leq J(\tilde\bz_{[1,N]}(k+1), \tilde\bu_{[0,N-1]}(k+1),z_{1|k}^\ast,1)-V(z_{0|k}^\ast)\\
&=z_{N|k}^{\ast\top}((\Tilde{A}+\Tilde{B}K)^\top P(\Tilde{A}+\Tilde{B}K)-P+Q+K^\top R K)z_{N|k}^\ast\\
&-l(z_{0|k}^\ast,u_{0|k}^\ast)+l_r(\tilde z_{0|k+1})-l_r(z_{0|k}^\ast)\\
&\leq -z_{0|k}^{\ast\top}Q z_{0|k}^\ast + \lambda \|z_{1|k}^\ast-\varphi(\bx_\tini(k+1))\|^2,
\end{align*}
which yields the inequality in \eqref{eq:V:z}.
\end{proof}

Typically, robust stability results for MPC, see, e.g., \cite{lazar2013}, use continuity of the value function $V$ evaluated at the true nonlinear system state $\varphi(\bx_\tini(k))$ to obtain input-to-state stability results. However, in the case of \eqref{eq:3:1} this is not possible because $\bx_\tini(k)$ may not belong to the feasible set $\Xset_f(N)$. Hence, instead, we firstly establish ISS \cite{JiangWang01} for the interpolated state trajectory $z_{0|k}^\ast$ and then we use the properties of comparison functions \cite{compendium2014} to establish ISS of the lifted Koopman state based on the true nonlinear system measurements, i.e., $\varphi(\bx_\tini(k))$.

\begin{theorem}
Assume that there exists $\alpha_2\in\cK_\infty$ such that $V(z)\leq \alpha_2(\|z\|)$ for all $z\in\Xset_f(N)$. Suppose that Assumption~\ref{assm:TerminalCost_and_Set} holds. Then for all $\varphi(\bx_\tini(0))\in\Xset_z$ and $z_{1|-1}^\ast\in\Xset_z$ such that the corresponding $z_{0|0}^\ast\in\Xset_f(N)$, along the trajectories of the closed-loop system \eqref{eqn:state_space}-\eqref{eq:3:1} it holds that
\begin{equation}
\label{eq:ISS}
\|\varphi(\bx_\tini(k))\|\leq \beta(\|\varphi(\bx_\tini(0))\|,k)+\gamma(\|\be_{[0,k]}\|),\quad\forall k\geq 1,
\end{equation}
for some $\beta\in\cK\cL$ and $\gamma\in\cK$.
\end{theorem}
\begin{proof}
First we show that for all $z_{0|0}^\ast\in\Xset_f(N)$ there exist $\beta_z\in\cK\cL$ and $\gamma_z\in\cK$ such that
\begin{equation}
\label{eq:ISS:z}
\|z_{0|k}^\ast\|\leq \beta_z(\|z_{0|0}^\ast\|,k)+\gamma_z(\|\be_{[0,k]}\|),\quad\forall k\geq 1.
\end{equation}
Indeed, by the hypothesis, inequality \eqref{eq:V:z} and the fact that $V(z_{0|k}^\ast)\geq \alpha_Q(\|z_{0|k}^\ast\|)$ for all $k\in\Nset$, by standard ISS Lyapunov function arguments \cite{JiangWang01} we obtain \eqref{eq:ISS:z}.  Next notice that by \eqref{eq:3:1c} it holds that
\begin{equation}
\label{eq:ISS:nom}
z_{0|k}^\ast-\varphi(\bx_\tini(k))=\xi_k^\ast(z_{1|k-1}^\ast-\varphi(\bx_\tini(k)))=\xi_k^\ast e(k).
\end{equation}
Hence, it holds that
\begin{align*}
\|\varphi(\bx_\tini(k))\|&=\|\varphi(\bx_\tini(k))-z_{0|k}^\ast+z_{0|k}^\ast\|\\
&\leq\|z_{0|k}^\ast\|+\|\varphi(\bx_\tini(k))-z_{0|k}^\ast\|\\
&=\|z_{0|k}^\ast\|+\|z_{0|k}^\ast-\varphi(\bx_\tini(k))\|\\&\leq\|z_{0|k}^\ast\|+\|e(k)\|\\
&\leq \beta_z(\|z_{0|0}^\ast\|,k)+\gamma_z(\|\be_{[0,k]}\|)+\|e(k)\|.
\end{align*}
Then due to $\beta_z\in\cK\cL$ it the follows that: 
\begin{align*}
\beta_z(\|z_{0|0}^\ast\|,k)&=\beta_z(\|z_{0|0}^\ast+\varphi(\bx_\tini(0))-\varphi(\bx_\tini(0))\|,k)\\
&= \beta_z(\|\varphi(\bx_\tini(0))+\xi_0^\ast e(0)\|,k)\\
&\leq \beta_z(\|\varphi(\bx_\tini(0))\|+\| e(0)\|,k)\\
&\leq \beta_z(2\|\varphi(\bx_\tini(0))\|,k)+\beta_z(2\| e(0)\|,k)\\
&\leq \beta_z(2\|\varphi(\bx_\tini(0))\|,k)+\beta_z(2\| e(0)\|,0).
\end{align*}
Above we used the fact that $\xi_k^\ast\in [0,1]$, the triangle inequality and standard properties of $\cK$ and $\cL$ functions. By combining the previous inequalities and using $\|e(i)\|\leq\|\be_{[0,k]}\|$ for any $i\in [0,k]$, we obtain \eqref{eq:ISS} with $\beta(s,k):=\beta_z(2s,k)\in\cK\cL$ and $\gamma(s):=\gamma_z(s)+s+\beta_z(2s,0)\in\cK$, which completes the proof.
\end{proof}
\begin{remark}\label{remark:ISS_noise}
The ISS bound \eqref{eq:ISS} holds in the presence of noisy data under certain assumptions, since ISS does not assume that the prediction error is bounded. 
\end{remark}

According to standard ISS arguments \cite{JiangWang01}, the ISS bound \eqref{eq:ISS} implies that when the Koopman model error tends to zero, the lifted measured state of the true nonlinear system converges to zero. Due to Assumption~\ref{assumption:observables} we further obtain that $\|\bx_\tini(k)\|$ converges to zero and, under detectability, that the true state of the nonlinear system converges to zero. 

Furthermore, the ISS bound on trajectories \eqref{eq:ISS} implies the finite asymptotic gain property \cite{JiangWang01}, which guarantees that all trajectories of the nonlinear system \eqref{eq:3:1} in closed-loop with the Koopman DPC control law obtained by solving \eqref{eq:3:1} are ultimately bounded in a set proportional with the Koopman model error. 
\begin{remark}[On the choice of $l_{r}(z_{0|k})$]
Given the equality in \eqref{eq:ISS:nom}, the regularization loss in \eqref{eq:3:lxi} can be replaced with 
\begin{align}
\label{eq:3:lxi:alt}
l_r(z_{0|k}):=\lambda\xi_k^2 \|e(k)\|_2^2, \quad \lambda>0,
\end{align}
which directly pushes $\xi_k$ to zero when $e(k)\neq 0$. Notice that if we were to use the original regularization cost proposed in \cite{kohler2022recursively}, i.e., $l_r(z_{0|k}):=\lambda\xi_k^2$, the inequality \eqref{eq:V:z} holds with $\sigma_r(\|e(k+1)\|)$ replaced by $\lambda$. 
Since $\lambda$ needs to be chosen sufficiently large, this only yields bounded trajectories in a very large set, which is not very useful. Also, this choice does not yield ISS for the interpolated state trajectory, but only input-to-state practical stability.  
\end{remark}

\section{Illustrative Example}\label{sec:IllustraticeExample}
To assess the effectiveness of the proposed scheme, we consider two nonlinear benchmark examples. The first example is a 2D cart-spring–damper (CSD) from \cite{raimondo2009min} and the second example is a nonlinear pendulum from \cite{dalla2022deep}. We analyze the accuracy of the learned observables by computing the $R^2$ coefficient for all the future prediction steps over the prediction horizon $j = 1, \dots, N$, i.e.,
\begin{equation}
    R^2 = 1 - \frac{\sum^{N_{test}}_{k=1} [y(k+j) -y_{j|k}]^2}{\sum^{N_{test}}_{k=1} [y(k+j) - \Bar{y}_j]^2},
\end{equation} 
where $N_{test}$ is the length of the test data set and $\Bar{y}_j = \frac{1}{N_{test}}\sum_{k=1}^{N_{test}}y(k+j)$. The training data set and test data set are shown in Appendix~\ref{app:2}. The proposed Koopman data-driven predictive control (KDPC) is compared with nonlinear model predictive control (NMPC). For NMPC, knowledge of the plants and state measurement is assumed while only input and output data is used for KDPC. For NMPC we use the input-state cost function $l(x_{i|k},u_{i|k}) := x_{i|k}^\top Qx_{i|k}+u_{i|k}^\top R u_{i|k}$ with terminal cost $l_N(x_{N|k}):= x_{N|k}^\top Px_{P|k}$ and the exact nonlinear model of the plant is used as prediction model and the NMPC problem is solved using FMINCON. More implementation details for Koopman data-driven predictive are discussed in Appendix~\ref{app:terminal}. The code can be found in the github repository\footnote{\url{https://github.com/todejong/KDPC_RSRFG}}
\paragraph{2D cart-spring–damper}\label{par:2chart}
To assess the effectiveness of the proposed scheme, we consider the discretized counterpart of the 2D cart-spring–damper (CSD) from \cite{raimondo2009min}, i.e.,     
\begin{figure}[t!]
  \centering
\includegraphics[scale=0.8,trim=0.0cm 0.0cm 0.0cm 0.0cm,clip]{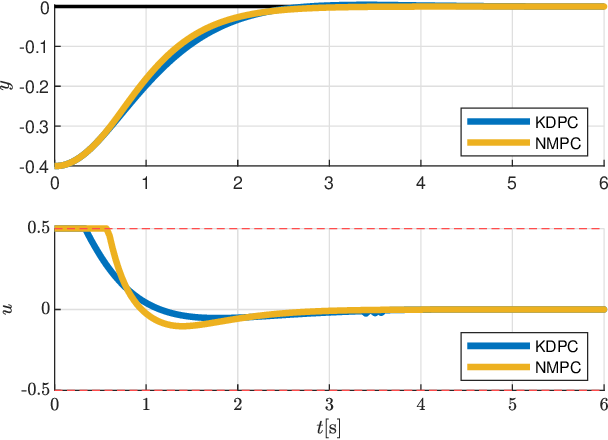}\vspace{-.2cm}
  \caption{Simulations chart: position $x_1$, control input $u$, interpolation variable $\xi$ and observables $\varphi_i$.}
  \label{fig:chart}
\end{figure}
\begin{figure}[t!]
  \centering
\includegraphics[scale=0.8,trim=0.0cm 0.0cm 0.0cm 0.0cm,clip]{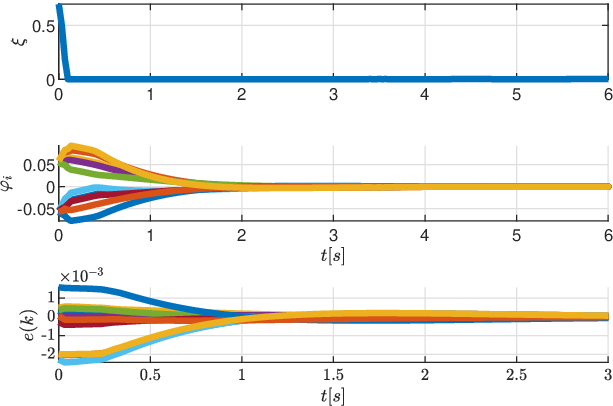}\vspace{-.2cm}
  \caption{Simulations chart: interpolation variable $\xi$, observables $\varphi_i$ and Koopman model error $e$.}
  \label{fig:chart}
\end{figure}
{\footnotesize
\begin{equation}\label{eqn:chart}
    \begin{split}
        x_1(k+1) &= x_1(k) + T_s x_2(k), \\
        x_2(k+1) &= x_2(k) - \frac{T_s k_0}{M} e^{-x_1(k)} x_1(k) - \frac{T_s \Bar{h}_d}{M} x_2(k) + \frac{T_s}{M}u(k), \\
        y(k) & = x_1(k),
    \end{split}
\end{equation}
}
considering the same parameters from \cite{raimondo2009min}. An open-loop identification experiment was performed to generate the output data using a multi-sine input constructed with the Matlab function idinput, with the parameters Range [$-4, 4$], Band [$0, 1$], Period $1000$, NumPeriod $1$ and Sine [$25, 40, 1$]. The observables are parameterized by means of a neural network with $8$ $\tanh$ activation functions and $2$ hidden layers and $T_{ini}=5$. The parameters are found by solving \eqref{eq:fit_basis} using the Adam optimizer with the MSE loss function from the Pytorch library with a learning rate of $1 \cdot 10^{-2}$. The resulting $R^2$ coefficient is computed, see Table~\ref{Tab:R2csd}. From this table it seems that the KDPC prediction model achieves accurate predictions on the test data with $R^2$ scores close to 1. Once the observables are learned the state data matrices $\textbf{Z}_p$ and $\textbf{Z}_f$ are found by feeding the same input-output data trough the learned observables. The prediction model matrices $\Psi^{LS}$ and $\Gamma^{LS}$ are then retrieved as in \eqref{eqn:LS_solution}. From them, we then extract the Koopman model matrices $\Tilde{A}$ and $\Tilde{B}$ as discussed at the end of Section~\ref{sec:koopmanModelIdentification} to construct the terminal set by solving \eqref{eq:3:2}. This can be done using standard techniques for linear MPC, i.e., first solve a corresponding linear matrix inequality for $(\tilde A, \tilde B)$ that yields a solution $(P,K)$ to \eqref{eq:3:2a} and then compute a polyhedral constraints admissible and invariant terminal set for $(\tilde A+\tilde B K)$ that satisfies \eqref{eq:3:2b} (e.g., using the MPT~3 Matlab Toolbox). The state of the system~\eqref{eqn:chart} is then initialized close to the constraints and stabilized using the Koopman DPC scheme in \eqref{eq:3:1} with the tuning parameters reported in Table~\ref{Tab:Controller_CSD}. More details about the implementation for the terminal cost and constraints are given in Appendix~\ref{app:terminal}.

We initialize the system at $x_1(0)=-0.4$m and $x_2(0)=0$m/s. The attained results are shown in \figurename{\ref{fig:chart}}. It is clear that the Koopman DPC controller remains recursively feasible and stabilizes the true nonlinear system while input constraints are active. Moreover, the interpolation variable is pushed to zero as soon as the input constraint becomes inactive, which implies that when this condition holds the lifted measured input-output data $\bx_\tini(k)$ are used and, hence, they are feasible. The results show that the observables converge to the origin, as expected when the prediction error vanishes. This is also observed in the third subplot in Figure~\ref{fig:chart}. Here it is shown that the Koopman model error defined by \eqref{eq:3:eK} is of the order $10^{-3}$ and goes to zero after initialization. From this and the
ISS property, we can conclude asymptotic stability of the Koopman state. All methods are able to stabilize the system while accounting for constraints. These simulations show that the KDPC controller achieves very similar performance compared with NMPC without any assumed knowledge of the plant.
\begin{table}[t!]
\centering
\caption{$R^2$ coefficients for CSD.}\label{Tab:R2csd}\vspace{-.2cm}
\resizebox{\columnwidth}{!}{%
\begin{tabular}{lccccccccccccc}
	$N$ & $1$ &  $3$& $5$ & $7$ & $9$ & $11$ & $13$ & $15$\\
	\hline
 \textbf{KDPC} & $0.9844$ & $0.9879$ & $0.9902$ & $0.9921$ & $0.9936$ & $0.9947$ & $0.9953$ & $0.9958$ \\
\end{tabular}\vspace{-.2cm}
}
\end{table}  
  
\begin{table}[t!]
\centering
\caption{Controller parameters and constraints CSD.}\label{Tab:Controller_CSD}\vspace{-.2cm}
\begin{tabular}{lcccccc}
	 & $Q$ & $R$ & $\lambda$ & $u_{min}$ & $u_{max}$ & $N$\\
	\hline
	\textbf{KDPC} & $10I_L$ & $1$ & $10^9$ & $-0.5$ & $0.5$ & $15$ \\
	\hline
 \textbf{NMPC} & $10I_n$ & $1$ & - & $-0.5$ & $0.5$ & $15$ \\
	\hline
\end{tabular}\vspace{-.2cm}
\end{table}


\paragraph{Pendulum}\label{par:pendulum}
To assess the effectiveness of the proposed scheme, we secondly consider the discretized nonlinear pendulum model from \cite{dalla2022deep}, i.e.,
\begin{equation}
    \begin{split}
        \begin{bmatrix}
            x_1(k+1) \\
            x_2(k+1)     
        \end{bmatrix} &= 
        \begin{bmatrix}
            1-\frac{b T_s}{J} & 0 \\
            T_s & 1    
        \end{bmatrix}  \begin{bmatrix}
            x_1(k) \\
            x_2(k)     
        \end{bmatrix} +  \begin{bmatrix}
            \frac{T_s}{J} \\
            0    
        \end{bmatrix} u(k) \\
         &- 
        \begin{bmatrix}
            \frac{M L g T_s}{2 J} \sin(x_2(k)) \\
             0
        \end{bmatrix}, \\
        y(k) &= x_2(k),
    \end{split}
\end{equation}
where $u(k)$ and $y(k)$ are the system input torque and
pendulum angle at time instant $k$, while $J = \frac{ML^2}{3}$, $M = 1$kg and $L = 1$m are the moment of inertia, mass and length of the pendulum. Moreover, $g = 9.81$m/s$^2$ is the gravitational acceleration, $b = 0.1$ is the friction coefficient and the sampling time $T_s = 1/30$s.
\begin{figure}[b!]
  \centering
\includegraphics[scale=0.8,trim=0.0cm 0.0cm 0.0cm 0.0cm,clip]{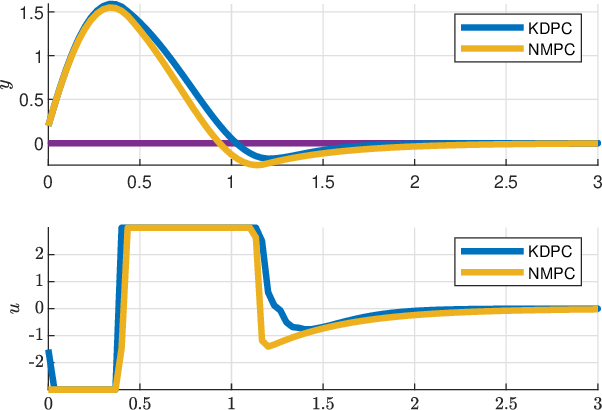}\vspace{-.2cm}
  \caption{Simulations pendulum: angle $x_2$, control input $u$, interpolation variable $\xi$ and observables $\varphi_i$.}
  \label{fig:pendulum}
\end{figure}
\begin{figure}[t!]
  \centering
\includegraphics[scale=0.8,trim=0.0cm 0.0cm 0.0cm 0.0cm,clip]{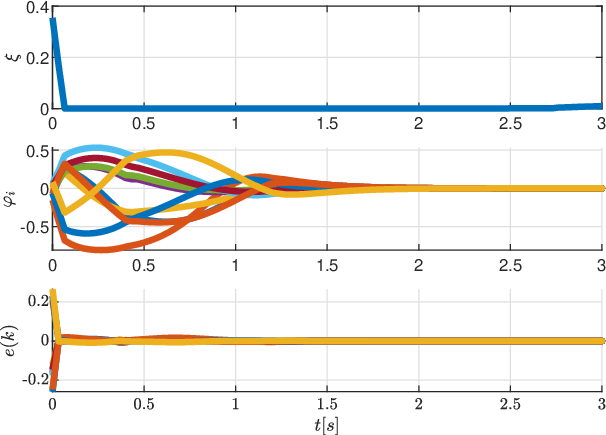}\vspace{-.2cm}
  \caption{Simulations pendulum: interpolation variable $\xi$, observables $\varphi_i$ and Koopman model error $e$.}
  \label{fig:pendulum}
\end{figure}
The output data is again generated by an open-loop identification experiment using a multi-sine input with the exact same parameters as for the 2D cart-spring–damper, see Appendix~\ref{app:2} for the test data and training data. The observables are parameterized by means of a neural network with $2$ hidden layers with $10$ $\tanh$ activation functions per layer and $T_{ini}=3$. More details about the implementations and the terminal cost and constraints are given in Appendix~\ref{app:terminal}.

We initialize the pendulum at $x_1(0)=7$ rad/s and $x_2(0)=0.2$ rad. The attained results are shown in \figurename{\ref{fig:pendulum}}. It is clear that all methods remains recursively feasible and stabilize the true nonlinear system while input constraints are active. The performance of the KDPC controller is again close to the one of NMPC. Just as for the other example, the interpolation variable is pushed to zero as soon as the input constraint becomes inactive and the observables converge to the origin. This is also observed in the third subplot in \figurename{\ref{fig:pendulum}}. Here it is shown that the Koopman model error defined by \eqref{eq:3:eK} quickly goes to zero after initialization. From this and the ISS property, we can conclude asymptotic stability of the Koopman state.

\begin{table}[t!]
\centering
\caption{$R^2$ coefficients for pendulum.}\label{Tab:Controller}\vspace{-.2cm}
\resizebox{\columnwidth}{!}{%
\begin{tabular}{lccccccccccccc}
	$N$ & $1$ &  $3$& $5$ & $7$ & $9$ & $11$ & $13$ & $15$\\
	\hline
 \textbf{KDPC} & $0.9139$ & $0.9109$ & $0.9070$ & $0.9016$ & $0.8960$ & $0.8899$ & $0.8822$ & $0.8717$ \\
\end{tabular}\vspace{-.2cm}
}
\end{table}   
    
\begin{table}[t!]
\centering
\caption{Controller parameters and constraints Pendulum.}\label{Tab:Controller}\vspace{-.2cm}
\begin{tabular}{lcccccc}
	 & $Q$ & $R$ & $\lambda$ & $u_{min}$ & $u_{max}$ & $N$\\
	\hline
	\textbf{KDPC} & $10I_L$ & $0.01$ & $10^9$ & $-3$ & $3$ & $15$ \\
	\hline
 \textbf{NMPC} & $10I_n$ & $0.01$ & - & $-3$ & $3$ & $15$ \\
	\hline
\end{tabular}\vspace{-.2cm}
\end{table}

\section{Conclusions and Future Works}
We have introduced a novel Koopman DPC scheme aimed at stabilizing general nonlinear systems. Our methodology involves a two-step approach. Firstly, we focus on learning the observables. Subsequently, we develop a best-fit state space model within the lifted Koopman space, that simplifies to a least squares problem in the lifted space. While our use of the linear-in-control input Koopman formulation enables the control problem to be framed as a quadratic program, such formulation is still incomplete. 

To address this limitation, we have devised an interpolated initial condition strategy for the Koopman DPC problem, ensuring recursive feasibility even in the presence of plant-model mismatches. Furthermore, we established sufficient conditions under which Input-to-State Stability is guaranteed. Through a simulation example, we demonstrated the effectiveness of our Koopman DPC approach in stabilizing nonlinear systems, thereby showcasing its potential.

\subsection{Future Works}
In our paper, we assumed a noiseless environment for both learning the Koopman model and conducting robust stability analysis. While our results establish recursive feasibility and robust stability for non-exact prediction models, they have not been extensively tested in the presence of noise. However, the interpolated initial condition ensures recursive feasibility if initially feasible even in the case of noise. The ISS bound also remains valid because it does not require that the prediction error is bounded. Future work should focus on deriving stochastic robust stability guarantees tailored to specific noise assumptions.
 
The performance of our proposed control scheme is inherently linked to the accuracy of the prediction model. However, it has been shown that the linear-in-control input Koopman formulation utilized in this paper may introduce prediction errors \cite{iacob2024koopman}. To enhance the performance of our control approach, it is recommended to extend our methodology by utilizing the complete Koopman formulation \cite{iacob2024koopman}. By addressing this incompleteness, we can potentially improve the accuracy and robustness of our control scheme. 


\newpage
\appendices
\section{Computation of Terminal Ingredients}\label{app:terminal}
The terminal cost $P$ and terminal constraints in the examples are computed by firstly extracting $\Tilde{A}$ and $\Tilde{B}$ from $\Psi^{LS}$ and $\Gamma^{LS}$ in \eqref{eqn:LS_solution}. Finding $P$ amounts to solving inequality~\eqref{eq:3:2a}. This is done by defining the variables $Y=K P^{-1}$ and $O= P^{-1}$, after which Schur complement is applied to obtain the LMI:
\begin{equation}
    \begin{split}
        &\begin{bmatrix}
            O & (\Tilde{A}O+\Tilde{B}Y)^\top & O & Y^\top \\
            (\Tilde{A}O+\Tilde{B}Y) & O & 0 & 0 \\
            O & 0 & Q^{-1} & 0 \\
            Y & 0 & 0 & R^{-1}
        \end{bmatrix} \succ 0, \\
        &O = O^\top, \quad O \succ 0. 
    \end{split}
\end{equation}
This LMI is solved by replacing $\succ 0$ by  $\succeq I10^{-6}$ and using the MOSEK solver. The terminal cost is found by $P=O^{-1}$ and the terminal control matrix is given by $K=Y O^{-1}$. Finding a terminal set $\Xset_T:=\{z \,: \, M_Nz\leq b_N\}$ amounts to computing a constraints admissible and invariant set for the dynamics $(\tilde A+\tilde B K)$. This is done by defining the lifted state constraint sets $\Xset := \{ z\,: \, A_z z\leq b_z\}$, the input constraint set $\Uset := \{ u\,: \, A_u u\leq b_u\}$ and the admissible constraint set $\Xset\Uset := \big\{ \{z\,: \, A_u K z \leq b_u\}\ \cap \Xset \big\}$ using the Polyhedron function from the MPT 3 toolbox. The terminal constraint set $\Xset_T$ for the dynamics $(\tilde A+\tilde B K)$ is found using the invariantSet function from the MPT 3 toolbox. 

The following MATLAB commands are used to find $M_N$ and $b_N$, i.e.,
\begin{enumerate}
    \item model = LTISystem('$A$', $\Tilde{A}+\Tilde{B}K$)
    \item Xset = Polyhedron($A_x$,$b_x$)
    \item Uset = Polyhedron($A_u$, $b_u$)
    \item XUset = Polyhedron($A_u K$,$b_u$) \& Xset
    \item InvSet = model.invariantSet('X',XUset)
\end{enumerate}
\newpage
\section{Identification data}\label{app:2}
\begin{figure}[h!]
  \centering
\includegraphics[scale=0.65,trim=0.0cm 0.0cm 0.0cm 0.0cm,clip]{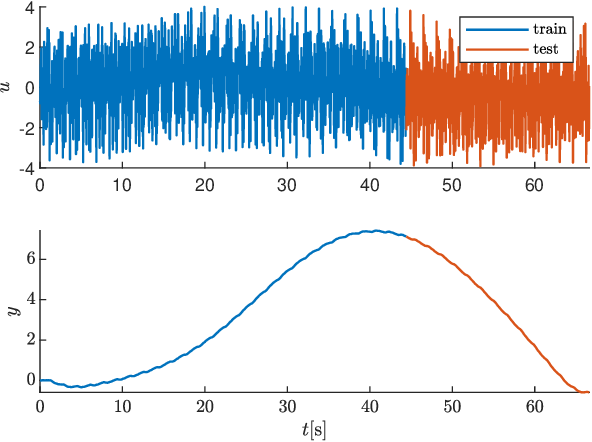}\vspace{-.2cm}
  \caption{Input-output data for 2D chart-spring damper.}
  \label{fig:IO_data_chart}
\end{figure}

\begin{figure}[h!]
  \centering
\includegraphics[scale=0.65,trim=0.0cm 0.0cm 0.0cm 0.0cm,clip]{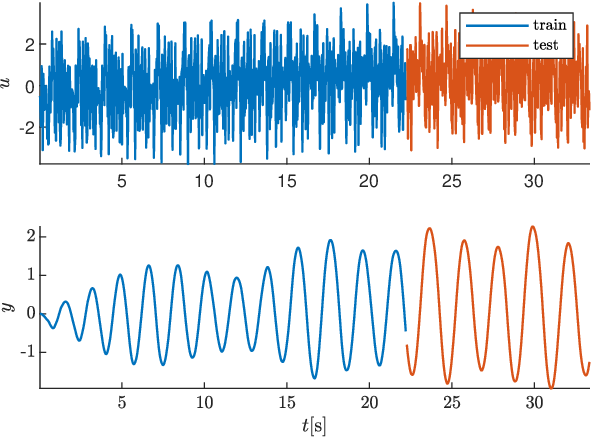}\vspace{-.2cm}
  \caption{Input-output data for Pendulum.}
  \label{fig:IO_data_pendulum}
\end{figure}
\end{document}